\documentclass[a4paper, 8pt]{article}
\usepackage[utf8]{inputenc}
\usepackage{a4wide}

\usepackage{amsmath, amssymb, amsthm, amsfonts, mathtools, mathrsfs, hyperref}

\usepackage{dsfont}

\usepackage[pdftex]{graphicx}

\usepackage{fancyhdr}

\usepackage{setspace}

\usepackage{tikz-cd}

\usepackage{bbm}

\usepackage{float}
\restylefloat{table}
\restylefloat{figure}

\usepackage{appendix}

\usepackage[colorinlistoftodos]{todonotes}

\DeclareMathOperator{\sgn}{sgn}
\DeclareMathOperator{\Tr}{Tr}
\def\d{\;{\rm d}}

\theoremstyle{plain}
\newtheorem{theorem}{Theorem}[section]
\newtheorem{proposition}[theorem]{Proposition}
\newtheorem{lemma}[theorem]{Lemma}
\newtheorem{corollary}[theorem]{Corollary}

\theoremstyle{definition}

\newtheorem{example}[theorem]{Example}

\newtheorem{remark}[theorem]{Remark}

\title{Convergence of the logarithm of the characteristic polynomial of unitary Brownian motion in Sobolev space}
\author{Johannes Forkel\footnote{\texttt{johannes.forkel@maths.ox.ac.uk}, Mathematical Institute, University of Oxford}, Isao Sauzedde\footnote{\texttt{isao.sauzedde@warwick.ac.uk}, Department of Statistics, University of Warwick }}

\date{\today}

\begin{document}

\maketitle

\begin{abstract}
    We prove that the convergence of the real and imaginary parts of the logarithm of the characteristic polynomial of unitary Brownian motion toward Gaussian free fields on the cylinder, as the matrix dimension goes to infinity, holds in certain suitable Sobolev spaces, which we believe to be optimal. This is the natural dynamical analogue of the result for a fixed time by Hughes, Keating and O'Connell \cite{HughesKeatingOConnell}.
    A weak kind of convergence is known since the work of Spohn \cite{Spohn},    which was widely improved recently by Bourgade and Falconet         \cite{BourgadeFalconet}. 
    In the course of this research we also proved a Wick-type identity, which we include in this paper, as it might be of independent interest.
\end{abstract}

\section{Introduction}

As unitary Brownian motion preserves the Haar measure on the unitary group $U(n)$, to many results of Haar distributed unitary matrices there is a corresponding dynamical result for a unitary Brownian motion $U$ at equilibrium. This is in particular the case for some properties of the eigenvalues, whose dynamics have been studied first by Dyson \cite{Dyson}, who computed a stochastic differential equation describing their evolution. In this paper, we intend to achieve such a transition from static to dynamic for the Hughes-Keating-O'Connell theorem on the large $n$ limit of the logarithm $\log p_n$ of the characteristic polynomial.\\

Characteristic polynomials of random matrices are fundamental objects in random matrix theory. They are closely related to the theory of log-correlated fields and to Gaussian multiplicative chaos \cite{HughesKeatingOConnell, BourgadeFalconet, Webb2}. In the case of Haar-distributed matrices from the classical compact groups, there are also remarkable similarities between the statistics of the characteristic polynomial and those of the Riemann zeta function and other number-theoretic $L$-functions, which led to a number of very precise conjectures for those $L$-functions \cite{Conreyetall, FyodorovHiaryKeating, FyodorovKeating, KeatingSnaith, KeatingSnaith2} - see \cite{BaileyKeating2} for a review.\\


The real and imaginary part of the logarithm $\log p_n$ of the characteristic polynomial also enters the wide family of \emph{linear statistics} of the eigenvalues $\lambda_1,\dots, \lambda_n$, that is functions that can be expressed as $ \sum_{i=1}^n f(\lambda_i)$. This family has received much attention already, both in the static \cite{Diaconis, DiaconisSha,Dobler,johansson2020} and dynamical \cite{Spohn, Webb} frameworks. Except for \cite{Diaconis} in the static case, all these papers assume too much regularity on $f$ to be applicable directly to $\log p_n$, for the type of convergence they use is too strong. However, it is still possible to use the result of Spohn \cite{Spohn} to identify the large $n$ limit of $\Re \log p_n$ and $\Im \log p_n$ as Gaussian free fields, and prove a weak type of convergence (see Lemma \ref{lemma:finite dimensionial} below or the remark (i) below Theorem 1 in \cite{Spohn} ). \\


In a recent paper \cite{BourgadeFalconet}, Bourgade and Falconet gave the first dynamical extension of Fisher-Hartwig asymptotics. Those asymptotics allowed them to give a new proof and improvement of Spohn's result. They also used those asymptotics to prove that $|p_n|^\alpha$, for certain $\alpha$ and when properly normalized, converge to a Gaussian multiplicative chaos measure associated to the Gaussian free field $h$ on the cylinder, i.e. informally the exponential of a multiple of $h$.\\ 

The goal of this paper is to specify some Sobolev spaces, which we think to be optimal, in which we prove the convergence of $\Re \log p_n$ and $\Im \log p_n$. \footnote{Several Sobolev spaces are involved because we can improve the regularity with respect to one of the parameters at the cost of sacrificing some regularity with respect to the other parameter.} This is the natural dynamical version of the corresponding stationary result for Haar-distributed unitary matrices by Hughes, Keating and O'Connell \cite{HughesKeatingOConnell}, who proved that for any fixed time the logarithm of the characteristic polynomial converges to a generalized Gaussian field on the unit circle.\\

In the last section, we state and prove an identity that allows to express the second moment of the trace of arbitrary products of a GUE matrix $H$ and an independent CUE matrix $U$ in terms of moments of $U$ only. 
When the dimension $n$ is large enough, the Diaconis-Shahshahani theorem on moments of traces of unitary matrices \cite{DiaconisSha} allows to then compute this new expression explicitly as a polynomial in $n$.

\subsection{Context}
We let $U_n:[0,\infty) \rightarrow U(n)$ be a unitary Brownian motion started from Haar measure (for a precise definition see Section \ref{section:UBM}), and define its characteristic polynomial as
\begin{equation*}
p_n(t,\theta): = \text{det}\left(I_n-e^{-i\theta}U_n(t)\right) = \prod_{k=1}^n (1-e^{i(\theta_k(t)-\theta)}), \quad (\theta, t) \in [0,2\pi) \times [0,\infty),
\end{equation*}
where $0 \leq \theta_1(t) < ... < \theta_n(t) < 2\pi$ denote the eigenangles of unitary Brownian motion. We define its logarithm by
\begin{align*}
\begin{split}
\log p_n(t,\theta):=& \sum_{k = 1}^n \log (1-e^{i(\theta_k(t) -\theta)}),
\end{split}
\end{align*}
with the branches on the RHS being the principal branches, such that
\begin{equation*}
\Im \log (1-e^{i(\theta_k(t) - \theta)}) \in \left( - \frac{\pi}{2}, \frac{\pi}{2} \right],
\end{equation*}
with $\Im \log 0 := \pi/2$. \\

Hughes, Keating and O'Connell proved that for any fixed time $t \geq 0$, $\log p_n(t, \cdot)$ converges to a generalized Gaussian field. Their result, reformulated to our setting, is as follows:

\begin{theorem} [Hughes, Keating, O'Connell \cite{HughesKeatingOConnell}]
    For any $\epsilon > 0$ and any fixed $t \geq 0$, the sequence of random functions $\left( \log p_n(t, \cdot) \right)_{n \in \mathbb{N}}$ converges in distribution in $H_0^{-\epsilon}(S^1)$ to the generalized Gaussian field
    \begin{align*}
        X(\theta) =& \sum_{k = 1}^\infty \frac{A_k}{\sqrt{k}} e^{ik\theta},
    \end{align*}
    where $A_k$ is a complex Gaussian whose real and imaginary parts are independent centered Gaussians with variance $1/(2k)$.
\end{theorem}

It is natural thus to assume that in the dynamic case, i.e. when considering $\log p_n$ also as a function of $t$, that the limit (in an appropriate function space) would be given by
\begin{align} \label{eqn:X}
    X(t,\theta) = \sum_{k = 1}^\infty \frac{A_k(t)}{\sqrt{k}} e^{ik\theta},
\end{align}
where $A_k(\cdot)$, $k \in \mathbb{N}$, are independent complex Ornstein-Uhlenbeck processes started from their stationary distribution, i.e. (up to a linear time change) solutions to the SDEs
\begin{align}
    \label{eqn:SDEA}
    \text{d}A_k(t) = -k A_k(t) \text{d}t + \text{d}\left(W_k(t) + i\tilde{W}_k(t) \right),
\end{align}
with $A_k(0)$ being a complex Gaussian whose real and imaginary parts are independent Gaussians with variance $1/(2k)$, and $(W_k(t))_{t \geq 0}$, $(\tilde{W}_k(t))_{t \geq 0}$, $k \in \mathbb{N}$, denoting real standard Brownian motions. \\

Our main result proves precisely that (for a definition of the Sobolev spaces $H^{s}([0,T])$ and $H^{-\epsilon}_0(S^1)$ see Section \ref{section:Sobolev}):

\begin{theorem} [Main Result] \label{thm:main}
For any $s \in (0,\frac{1}{2})$, $\epsilon > s$ and $T>0$, the sequence of random fields $ \left( \log p_n (\cdot, \cdot) \right)_{n \in \mathbb{N}}$ converges in distribution in the tensor product of Hilbert spaces $H^{s}([0,T]) \otimes  H^{-\epsilon}_0(S^1)$ to the generalized Gaussian field $X$ in (\ref{eqn:X}).
\end{theorem}

A calculation shows that the covariance functions of $\Re X$ and $\Im X$ are given by
\begin{align*}
    \mathbb{E}(\Re X(t,\theta), \Re X (t',\theta')) &= \mathbb{E}(\Im X(t,\theta), \Im X (t',\theta'))  = \frac{1}{2} \log \frac{\max (e^{-t}, e^{-t'})}{|e^{-t}e^{i\theta} - e^{-t'}e^{i\theta'}|}.
\end{align*}
The centered Gaussian fields $\Re X$ and $\Im X$ with such a covariance function have been identified as Gaussian free fields on the infinite cylinder $\mathbb{R} \times \mathbb{R} / 2\pi \mathbb{Z}$ in \cite[Section 2.2]{BourgadeFalconet}.




\begin{remark}
    Theorem \ref{thm:main} implies that there is a trade-off between regularity in $\theta$ and regularity in $t$. We believe that the regularity we obtain is optimal, in the sense that for $s = 1/2$ or $\epsilon = s$, $X$ is almost surely not an element of the tensor product of $H^{s}([0,T])\otimes  H^{-\epsilon}_0(S^1)$ anymore.

    While the limiting field is rotationally invariant from an infinitesimal point of view, this is not the case for $\log p_n$ with finite $n$. In particular, one can exchange the regularity in the variable $t$ with the regularity in the variable $\theta$ for the limiting field, but for our proof of convergence to work, the Sobolev regularity $-\epsilon$ in the variable $\theta$ needs to be negative which is not the case for the Sobolev regularity $s$ in the variable $t$.
\end{remark}

Just like in the stationary case, the Gaussian field $X$ can't be defined pointwise as its variance at each point is infinite, but it can still be "exponentiated" to build a Gaussian multiplicative chaos (GMC) measure. When we let $h(t,\theta)$ denote the real part of $X(t,\theta)$, and denote by $h_\delta(t,\theta)$ a mollification of $h$, then for $\gamma \in (0,2\sqrt{2})$ the random measures
\begin{align*}
    e^{\gamma h(t,\theta)} \text{d}\theta \text{d}t := \lim_{\delta \rightarrow 0} e^{\gamma h_\delta (t,\theta) - \frac{\gamma^2}{2}\mathbb{E}(h_\delta(t,\theta))} \text{d}\theta \text{d}t
\end{align*}
exist and are non-trivial, where the limit is in probability w.r.t. the topology of weak convergence of measures on $\mathbb{R} \times \mathbb{R} / 2\pi \mathbb{Z}$, see \cite{Berestycki} for a self-contained proof of this fact. Bourgade and Falconet proved that exponentiating $\log |p_n(t,\theta)|$ in this way, and then taking the large $n$ limit, gives the same limiting measure as when first taking the large $n$ limit to obtain the Gaussian free field $h$, and then exponentiating it. Their result is the dynamical analogue to Webb's result for fixed $t$ and the measures being on the unit circle \cite{Webb}, and its precise statement is as follows:

\begin{theorem}[Bourgade, Falconet \cite{BourgadeFalconet}]
\label{thm:bourgadeFalconet}
    For every $\gamma \in (0,2\sqrt{2})$ it holds that
    \begin{align*}
        \lim_{n \rightarrow \infty} \frac{|p_n(t,\theta)|^\gamma}{\mathbb{E}\left( |p_n(t,\theta)|^\gamma \right)} \text{d}\theta \text{d}t = e^{\gamma h(t,\theta)} \text{d}\theta \text{d}t,
    \end{align*}
    where the convergence is in distribution in the space of Radon measures on the infinite cylinder $\mathbb{R} \times \mathbb{R} / 2\pi \mathbb{Z}$, equipped with the topology of weak convergence.
\end{theorem}
Our main result complements their asymptotics in that it shows in which Sobolev spaces the convergence of the underlying fields $\log |p_n|$ and $\Im \log p_n$ to the Gaussian free field $h$ holds.\\

Further, Theorem \ref{thm:main} is related to the below result by Spohn, which we also use in our proof. 
For real-valued functions $f \in H_0^{3/2 + \epsilon}(S^1, \mathbb{R})$, $\epsilon > 0$, Spohn considered linear statistics of the eigenvalues $e^{i\theta_1(t)},...,e^{i\theta_n(t)}$ of unitary Brownian motion (in fact he more generally considered interacting particles on the unit circle with different repulsion strengths):
\[
    \xi_n(t,f): =\sum_{j=1}^n f(e^{i\theta_j(t)}), \qquad (t,f) \in [0,\infty) \times H^{3/2+\epsilon}_0(S^1, \mathbb{R}).
\]
Since $H_0^{-3/2 - \epsilon}(S^1, \mathbb{R})$ is the dual space of $H_0^{3/2 + \epsilon}(S^1, \mathbb{R})$, one can consider $\xi_n$ as a random continuous map $t \mapsto \xi_n(t, \cdot) \in H_0^{-3/2-\epsilon}(S^1, \mathbb{R})$.
\begin{theorem}[Spohn \cite{Spohn}] \label{thm:Spohn}
    For any $\epsilon>0$, as $n \rightarrow \infty$, $\xi_n(t,f)$ converges to a stationary solution of the SDE
    \[
    \d \xi(t,f) =\xi(t, -\sqrt{-\partial_\theta^2} f) \d t+\d \mathcal{W}(t,f'),
    \]
    where $\d \mathcal{W}$ is a white noise given by
    \[\mathbb{E}[ \d \mathcal{W}(t,f) \d \mathcal{W}(s,g)]=2 \delta(t-s) \d s\d t \frac{1}{2\pi}\int_0^{2\pi} f(e^{i\theta})g(e^{i\theta})\d \theta, \]
    and where the convergence is in distribution in $\mathcal{C}(\mathbb{R}, H^{-3/2-\epsilon}(S^1, \mathbb{R} ))$, endowed with the topology of locally uniform convergence. The stationary distribution is given by a Gaussian with covariance
    \begin{align*}
        \mathbb{E} \left( \xi(t,f) \xi(t,g) \right) = \sum_{k \neq 0} |k| f_k g_k.
    \end{align*}
\end{theorem}
Here, $\sqrt{-\partial_\theta^2} f$ is simply the function whose $j^{\text{th}}$ Fourier coefficient is $|j|$ times the $j^{\text{th}}$ Fourier coefficient of $f$.
This result shows in particular that the $k^{\text{th}}$ Fourier coefficient of $\ln p_n$ converges toward $\frac{A_k}{\sqrt{k}}$ (see Lemma \ref{lemma:finite dimensionial} below).    
Further, during the proof of Theorem \ref{thm:main}, we will need the following result from Bourgade and Falconet \cite[Corollary 3.2]{BourgadeFalconet}:

\begin{corollary}[Bourgade, Falconet] \label{thm:multi time}
    Let $(z_1(t),...,z_n(t))_{t \geq 0}$ denote the eigenvalue process of unitary Brownian motion, started at Haar measure, and denote $\sgn(x) = 1_{x > 0} - 1_{x < 0}$. For $f,g \in H_0^{1/2}(S^1, \mathbb{R})$, we have for every $n \in \mathbb{N}$ and $t \geq 0$,
    \begin{align*}
        &\mathbb{E} \Big[ \Big( \sum_{j = 1}^n f(z_j(0)) \Big) \Big( \sum_{j = 1}^n g(z_j(t)) \Big) \Big]
        = \hspace{-0.2cm}\sum_{|k| \leq n - 1} \hspace{-0.2cm} f_k g_{-k} \sgn (k) e^{-|k|t} \frac{\sinh (\frac{k^2t}{n})}{\sinh (\frac{kt}{n})} + \sum_{|k| \geq n} f_k g_{-k} e^{-\frac{k^2t}{n}} \frac{\sinh (kt)}{\sinh (\frac{kt}{n})}.
    \end{align*}
\end{corollary}

\section{Mathematical Preliminaries}
\subsection{Unitary Brownian motion} \label{section:UBM}
Brownian motion $(U_n(t))_{t \geq 0}$ on the unitary group $U(n)$ is the diffusion governed by the stochastic differential equation
\begin{align*}
    \text{d}U_n(t) = \sqrt{2} U_n(t) \text{d}B_n(t) - U_n(t) \text{d}t,
\end{align*}
with $(B_n(t))_{t \geq 0}$ denoting a Brownian motion on the space of skew-Hermitian matrices. That is
\begin{align*}
    B_n(t) = \sum_{k = 1}^{n^2} X_k \tilde{B}^{(k)}(t),
\end{align*}
where $\tilde{B}^{(k)}$, $k = 1,...,n^2$, are independent one-dimensional standard Brownian motions, and where the matrices $X_k$, $k = 1,...,n^2$, are an orthonormal basis of the real vector space of skew-Hermitian matrices w.r.t. the scalar product $\langle A, B \rangle := n \Tr (AB^*)$. One such basis is given by the matrices $\frac{1}{\sqrt{2n}}(E_{k,l} - E_{l,k})$, $\frac{i}{\sqrt{2n}}(E_{k,l} + E_{l,k})$, $1 \leq k < l \leq n$, and $\frac{i}{\sqrt{n}} E_{k,k}$, $1 \leq k \leq n$.

\begin{remark}
    Unitary Brownian motion is usually defined using a different normalisation, i.e. satisfying the SDE $\d \tilde{U}_n(t) = \tilde{U}_n(t) \d B_n(t) - \frac{1}{2}\tilde{U}_n(t)$. With this normalisation the generator is given by one half times the Laplacian on $U(n)$, which is the usual definition of Brownian motion on a Riemannian manifold. The relation between the two normalisations is $\tilde{U}_n(2t) = U_n(t)$.
\end{remark}

In this paper we always consider unitary Brownian motion started from Haar measure on $U(n)$, which is its stationary distribution. Thus $U_n(t)$ is Haar distributed for all $t \geq 0$.

\subsection{Sobolov spaces and their Tensor Product} \label{section:Sobolev}

Consider the space of square integrable $\mathbb{C}$-valued functions on the unit circle, with vanishing mean:
\begin{equation*}
    L^2_0(S^1)= \left\{ f(\theta) = \sum_{k \in \mathbb{Z}} f_k e^{ik\theta} : \sum_{k \in \mathbb{Z}} |f_k|^2 < \infty, f_0=0 \right\}.
\end{equation*}

For $s \geq 0$, we define $H^s_0(S^1)$ as the restriction of $L^2_0(S^1)$ w.r.t. the functions for which the inner product
\begin{equation*}
    \langle f, g \rangle_s = \sum_{k \in \mathbb{Z}} |k|^{2s} f_k\overline{g_k}
\end{equation*}
is finite. For $s \leq 0$, we define $H^s_0(S^1)$ as the completion of $L^2_0(S^1)$ w.r.t. this scalar product. Note that $\left(H_0^s(S^1), \langle \cdot , \cdot \rangle_s \right)$ is a Hilbert space for all $s \in \mathbb{R}$. For $s \geq 0$ it is a subspace of $H^0_0(S^1) = L_0^2(S^1)$, i.e. the space of square-integrable functions with zero mean, while for $s < 0$, $H_0^s(S^1)$ can be interpreted as the dual space of $H_0^{-s}(S^1)$, i.e. as a space of generalized functions defined up to additive constant.\\

For $T>0$, and $s \in (0,1)$, we define the fractional Sobolev space $H^s([0,T])$ as the subspace of $L^2([0,T])$, where the Slobodeckij inner product
\begin{equation*}
(f, g)_s := \int_0^T f(t) \overline{g(t)} \text{d}t + \int_0^T\int_0^T \frac{(f(t)-f(u))(\overline{g(t)-g(u)} )}{|t-u|^{1+2s}} \d u\d t
\end{equation*}
is finite. Note that $(H^s([0,T]),(\cdot, \cdot)_s)$ is a Hilbert space for all $s > 0$.

\begin{remark}
For the fact that the fractional Sobolev spaces defined through Fourier series or through the Slobodeckij norm agree, the reader can consult e.g. \cite{Hitchhiker}.
\end{remark}

For $s > 0$ and $\epsilon > 0$ we let $H^s([0,T]) \otimes H^{-\epsilon}_0(S^1)$ denote the tensor product of Hilbert spaces $H^s([0,T])$ and $H^{-\epsilon}_0(S^1)$. Since the inner product on that space is determined by
\begin{align*}
\langle f\otimes g, h\otimes k\rangle_{s,-\epsilon}
&=
(f,h)_s \langle g, k\rangle_{-\epsilon}\\
&= \int_0^T f(t)\overline{h(t)} \text{d}t \langle  g, k \rangle_{-\epsilon}
+\int_0^T\int_0^T \frac{ (f(t) -f(u) )  \overline{(h(t) -h(u) )} }{|t-u|^{1+2s}} \d u\d t \langle  g, k \rangle_{-\epsilon}, \\
&= \int_0^T \langle f(t) g, h(t)k \rangle_{-\epsilon} \text{d}t
+\int_0^T\int_0^T \frac{\langle (f(t) -f(u) )g , (h(t) -h(u) )k \rangle_{-\epsilon}}{|t-u|^{1+2s}} \d u\d t,
\end{align*}
we obtain
\[
\langle F,G\rangle_{s,-\epsilon}=\int_0^T \langle F(t,\cdot),G(t,\cdot)\rangle_{-\epsilon} \d t+  \int_0^T\int_0^T \frac{\langle F(t,\cdot)-F(u,\cdot), G(t,\cdot)-G(u,\cdot)\rangle_{-\epsilon}}{|t-u|^{1+2s}} \d u\d t,
\]
first when $F$ and $G$ are linear combinations of pure tensor products, and then for all $F,G\in H^s([0,T])\otimes H^{-\epsilon}_0(S^1)$ by density and continuity.

\section{Proof of the main result Theorem \ref{thm:main}}

The proof strategy is as in the stationary case in \cite{HughesKeatingOConnell}: we treat $\left( \log p_n \right)_{n \in \mathbb{N}}$ as a sequence in $H^{s}([0,T])\otimes  H^{-\epsilon}_0(S^1) $, and show that if any of its subsequences has a limit then that limit has to be $X$. We do this by showing that the finite-dimensional distributions of $\left( \log p_n \right)_{n \in \mathbb{N}}$, i.e. the distributions of finite sets of Fourier coefficients at a finite number of times, converge to those of $X$. We then show that the set $\left( \log p_n \right)_{n \in \mathbb{N}}$ is tight in $H^s([0,T])\otimes  H^{-\epsilon}_0(S^1) $. Since $H^{s}([0,T])\otimes  H^{-\epsilon}_0(S^1) $ is complete and separable, Prokhorov's theorem implies that the closure of $\left(\log p_n\right)_{n \in \mathbb{N}}$ is sequentially compact w.r.t. the topology of weak convergence. In particular this means that every subsequence of $\left( \log p_n \right)_{n \in \mathbb{N}}$ has a weak limit $H^{s}([0,T])\otimes  H^{-\epsilon}_0(S^1) $. Since any such limit has to be $X$ it follows that the whole sequence $\left(\log p_n \right)_{n \in \mathbb{N}}$ must converge weakly to $X$.\\

We recall that
\begin{equation*}
\log (1 - z) = - \sum_{k = 1}^\infty \frac{z^k}{k}
\end{equation*}
for $|z| \leq 1$, where for $z = 1$ both sides equal $-\infty$. By using the identity $\log \det = \Tr \log$ we see that the Fourier expansion of $\log p_n$ w.r.t. the spacial variable $\theta$ is given as follows:
\begin{align*}
\begin{split}
\log p_n(t,\theta) =& - \sum_{k = 1}^{\infty} \frac{\Tr (U_n^k(t))}{k} e^{-ik\theta}.
\end{split}
\end{align*}

\begin{lemma}
    \label{lemma:finite dimensionial}
    Let $((\log p_n)_k(t))_{k\geq 1}$ be the Fourier coefficients of $(\log p_n)(t, \cdot )$.
    The process $(t,k) \mapsto (\log p_n)_k(t) $ converges in finite-dimensional distributions towards the complex Ornstein-Uhlenbeck process $(t,k)\mapsto A_k(t)$ defined in (\ref{eqn:SDEA}).
\end{lemma}
\begin{proof}
    We prove convergence of the finite-dimensional distributions by showing that for any $l \in \mathbb{N}$ and $0 \leq t_1 < t_2 < ... < t_l \leq T$, as $n \rightarrow \infty$:
    \begin{align*} \label{eqn:finiDim}
    &\Big( (\log p_n)_{1}(t_1), ..., (\log p_n)_{l}(t_1), (\log p_n)_{1}(t_2), ..., (\log p_n)_{l}(t_2), ..., (\log p_n)_{1}(t_l),..., (\log p_n)_{l}(t_l) \Big) \nonumber\\
    &\hspace{4cm}
    \overset{(d)}\longrightarrow  \Big(A_1(t_1),  ...,  A_l(t_1), A_1(t_2), ..., A_l(t_2), ..., A_1(t_l),  ..., A_l(t_l) \Big).
    \end{align*}
    Let $e_k:\theta\mapsto e^{i k\theta}$. Then, using the notations of Theorem \ref{thm:Spohn},  $\log(p_n)_k(t)=\xi_n\big(t,\tfrac{e_k}{k}\big)$. Thus Spohn's theorem, combined with the continuous mapping theorem with the appropriate continuous map $\mathcal{C}(\mathbb{R}, H^{-3/2-\epsilon}(S^1, \mathbb{R} )) \rightarrow \mathbb{R}^{l^2}$, implies that
    \begin{align*}
        &\Big( (\log p_n)_{1}(t_1), ..., (\log p_n)_{l}(t_1), (\log p_n)_{1}(t_2), ..., (\log p_n)_{l}(t_2), ..., (\log p_n)_{1}(t_l),..., (\log p_n)_{l}(t_l) \Big) \nonumber\\
        &\hspace{2cm}
        \overset{(d)}\longrightarrow  \Big( \xi\big(t_1, \tfrac{e_1}{1}\big),\dots, \xi\big(t_1, \tfrac{e_l}{l}\big), \xi\big(t_2, \tfrac{e_1}{1}\big),\dots, \xi\big(t_2, \tfrac{e_l}{l}\big), ...\xi\big(t_l, \tfrac{e_1}{1}\big),\dots, \xi\big(t_l, \tfrac{e_l}{l}\big)\Big).
    \end{align*}
    Combining the real and imaginary part of $e_k$, we obtain that the SDE for $\xi(\cdot, \frac{e_k}{k})$ reduces to
    \[
 \d  \xi \left(t, \frac{e_k}{k} \right)= -k \xi\left(t, \frac{e_k}{k}\right) \d t +  dB_k(t),
    \]
    where $B_k$ is a complex Brownian motion, i.e. a process whose real and imaginary parts are independent standard Brownian motions. Besides, the Brownian motions $(B_k)_{k\geq 0}$ are independent, so that $(A_k)_{k\geq 1}$ and $\big(\xi\big(\cdot, \tfrac{e_k}{k} \big)\big)_{k\geq 1}$ are equal in distribution, which concludes the proof.
\end{proof}

We proceed to show tightness of $\left( \log p_n \right)_{n \in \mathbb{N}}$ in $H^s([0,T]) \otimes H^{-\epsilon}_0(S^1)$, i.e. for every $\delta > 0$ we construct a compact $K_\delta \subset H^s([0,T]) \otimes H^{-\epsilon}_0(S^1)$ for which
\begin{equation*}
\sup_{n \in \mathbb{N}} \mathbb{P} \left( \log p_n \in K_\delta^c \right) < \delta.
\end{equation*}
We let $0 < s' < \epsilon'$ such that $0 < s < s' < \epsilon' < \epsilon$, and choose
\begin{align*}
    K_\delta = \left\{ F \in H^{s}([0,T])\otimes H_0^{-\epsilon}(S^1): ||F||^2_{s',-\epsilon'} \leq C_\delta \right\},
\end{align*}
for a $C_\delta$ depending on $\delta$. By Lemma \ref{lemma:compact inclusion} below we see that $K_\delta$ is compact in $H^{s}([0,T])\otimes H_0^{-\epsilon}(S^1)$, and by Lemma \ref{lemma:boundedness} below we see that $\sup_{n \in \mathbb{N}} \mathbb{E} \left( ||\log p_n||^2_{s',-\epsilon'} \right) < \infty$. Thus, when choosing $C_\delta$ large enough, we see that
\begin{align*}
    \sup_{n \in \mathbb{N}} \mathbb{P} \left( \left( \log p_n \right) \in K_\delta^c \right) =& \sup_{n \in \mathbb{N}} \mathbb{P} \left( || \log p_n ||_{s,-\epsilon}^2 > C_\delta \right) \\
    \leq& \frac{\sup_{n \in \mathbb{N}} \mathbb{E} \left( ||\log p_n||^2_{s',-\epsilon'} \right)}{C_\delta^2} \\
    <& \delta,
\end{align*}
which shows tightness of $\log p_n$ and thus together with Lemma \ref{lemma:finite dimensionial} proves our Theorm \ref{thm:main}.
\begin{lemma}\label{lemma:compact inclusion}
    Let $0 < s < s' < \epsilon' < \epsilon$. Then, the inclusion of $H^{s'}([0,T])\otimes H_0^{-\epsilon'}(S^1)$ into  $H^{s}([0,T])\otimes H_0^{-\epsilon}(S^1)$ is compact.
\end{lemma}
\begin{proof}
    From the Kondrachov embedding theorem, the inclusion $\iota_1$ of $H^{s'}([0,T)$ into $H^{s}([0,T])$ is compact, as well as the inclusion $\iota_2$ from $H_0^{\epsilon}(S^1)$ into $H_0^{\epsilon'}(S^1)$. Then the dual operator \\ $\iota_2^*: H_0^{-\epsilon'}(S^1) \rightarrow H_0^{-\epsilon}(S^1)$ is also compact. On Hilbert spaces, the tensor product of two compact operators is also compact (see e.g. \cite{Kubrusly} \footnote{In \cite{Kubrusly}, the result is stated for endomorphisms, but this extra assumption is not used in the proof.}), so that $\iota_1\otimes \iota_2^*$ is compact indeed.
\end{proof}

\begin{lemma}\label{lemma:boundedness}
    For all $s \in \big(0,\frac{1}{2}\big)$ and all $\epsilon > s$, it holds that $\sup_{n \in \mathbb{N}} \mathbb{E} \left( ||\log p_n||^2_{s,-\epsilon} \right) < \infty$.
\end{lemma}
\noindent \textbf{Proof:} We see that
\begin{align*}
    \mathbb{E} \left( ||\log p_n||^2_{s,-\epsilon} \right) =& \mathbb{E} \left( \int_0^T ||\log p_n(\cdot , t)||_{-\epsilon}^2 \text{d}t \right) + \mathbb{E} \left( \int_0^T \int_0^T \frac{||  \log p_n(\cdot, t) -  \log p_n(\cdot, r)||^2_{-\epsilon}}{|t-r|^{2s + 1}} \text{d}r \text{d}t \right).
\end{align*}
For the first summand it holds that (with $k \wedge n$ denoting $\min \{k,n\}$)
\begin{align*}
\begin{split}
\mathbb{E} \left( \int_0^T ||\log p_n(\cdot , t)||^2_{-\epsilon} \text{d}t \right) =& \int_0^T \mathbb{E} \left( \sum_{k=1}^\infty k^{-2\epsilon} \frac{ \left| \Tr (U_n(t)^k) \right|^2}{k^2} \right) \text{d}t \\
=& T \sum_{k = 1}^\infty k^{-2 - 2\epsilon} \mathbb{E} \left( \left| \Tr (U_n(0)^k) \right|^2 \right) \\
=& T \sum_{k = 1}^\infty k^{-2 - 2\epsilon} (k \wedge n) \\
< & T \sum_{k = 1}^\infty k^{-1-2\epsilon} < \infty.
\end{split}
\end{align*}
For the second summand it holds that:
\begin{align}
\label{eqn:split}
\begin{split}
    &\hspace{-2cm}\mathbb{E} \Big( \int_0^T \int_0^T \frac{||  \log p_n(\cdot, t) -  \log p_n(\cdot, r)||^2_{-\epsilon}}{|t-r|^{2s + 1}} \text{d}r \text{d}t \Big) \\
    =& \sum_{k = 1}^\infty k^{-2 - 2\epsilon} \int_0^T \int_0^T \frac{ \mathbb{E} \left( |\Tr (U_n^k(t) - U_n^k(r))|^2 \right) }{|t - r|^{2s+1}} \text{d}r \text{d}t\\
    \leq & CT\sum_{k = 1}^\infty k^{-2 - 2\epsilon} \int_0^T \frac{ \mathbb{E} \left( |\Tr (U_n^k(t) - U_n^k(0))|^2 \right) }{t^{2s+1}} \text{d}t \\
    \leq & CT\sum_{k = 1}^\infty k^{-2 - 2\epsilon} \int_0^{k^{-1}} \frac{ \mathbb{E} \left( |\Tr (U_n^k(t) - U_n^k(0))|^2 \right) }{t^{2s+1}} \text{d}t\\
    & + CT\sum_{k = 1}^\infty k^{-2 - 2\epsilon} \int_{k^{-1}}^{\infty} \frac{ 4 \mathbb{E} \left( |\Tr (U_n^k(0) )|^2 \right) }{t^{2s+1}} \text{d}t.
\end{split}
\end{align}
For the second summand in (\ref{eqn:split}) we get
\[
\int_{k^{-1}}^{\infty} \frac{ 4 \mathbb{E} \left( |\Tr (U_n^k(0) )|^2 \right) }{t^{2s+1}} \text{d}t
= 8 s (n\wedge k) k^{2s},
\]
which is sufficient since $\sum_{k = 1}^\infty k^{-2 - 2\epsilon+1+2s}$ is finite as soon as
$s < \epsilon $. \\

For the first sum in \eqref{eqn:split} we use Corollary \ref{thm:multi time}, which implies that for all $k \geq 1$
\begin{align*}
\begin{split}
    \mathbb{E}\left( \Tr (U_n^k(t)) \overline{\Tr (U_n^k(0))} \right) =& 1_{k < n} e^{-kt} \frac{ \sinh (\frac{k^2t}{n})}{\sinh (\frac{kt}n)} + 1_{k \geq n} e^{-\frac{k^2t}{n}} \frac{\sinh(kt)}{\sinh(\frac{kt}{n})} \\
    =& e^{-\frac{k(k \vee n) t}{n}} \frac{\sinh \left( \frac{k(k\wedge n)t}{n} \right)}{\sinh \left( \frac{kt}{n} \right)},
\end{split}
\end{align*}
with $k \vee n := \max \{k, n\}$. Using this, and the fact that $\sinh x \geq x$ and $1/\sinh x \geq 1/x - x/6$ for all $x > 0$, we see that for $t < k^{-1}$:
\begin{align*}
    &\mathbb{E} \left( |\Tr (U_n^k(t) - U_n^k(0))|^2 \right) \\
    =& \mathbb{E} \left( |\Tr (U_n^k(0))|^2 \right)+ \mathbb{E} \left( |\Tr (U_n^k(t))|^2 \right) - 2\mathbb{E} \left( \Tr (U_n^k(t)) \overline{\Tr (U_n^k(0))} \right) \\
    =& 2(k \wedge n) - 2e^{-\frac{k(k \vee n)t}{n}} \frac{\sinh \left( \frac{k(k\wedge n)t}{n} \right)}{\sinh \left( \frac{kt}{n} \right)} \\
    \leq & 2(k \wedge n) - 2e^{-\frac{k(k \vee n)t}{n}} \frac{k(k\wedge n)t}{n} \left( \left(\frac{kt}{n} \right)^{-1} -  \frac{kt}{6n} \right) \\
    = & 2(k \wedge n) - 2e^{-\frac{k(k \vee n)t}{n}} \left( k \wedge n - \frac{k^2t^2 (k \wedge n)}{6n^2} \right) \\
    =&  2e^{-\frac{k(k \vee n)t}{n}} \frac{k^2 t^2 (k \wedge n)}{6n^2} + 2(k \wedge n )(1 - e^{-\frac{k(k \vee n)t}{n}}) \\
    \leq & 2k^3t^2 + 2(k \wedge n) \frac{k(k \vee n)t}{n} \\
    \leq & 4k^2t.
\end{align*}
Thus we see that when $s < 1/2$, the first sum in \eqref{eqn:split} is bounded by $
    T \sum_{k = 1}^\infty k^{-1 - 2\epsilon + 2s}$,
which is finite for $s < \epsilon$. This finishes the proof. \qed \\

\section{A Wick-type identity} \label{section:Wick}

In this section we prove \ref{prop:orbits} below, which is about expectations of the form
\begin{align*}
&\mathbb{E} \Big( \Tr \left( H U^{\sigma_1} H U^{\sigma_2} \dots H U^{\sigma_j} \right) \overline{\Tr \left( H U^{\sigma_1} H U^{\sigma_2} \dots H U^{\sigma_j} \right)} \Big),
\end{align*}
where $\sigma_1,...,\sigma_j \in \mathbb{Z}$, $U \in U(n)$ is Haar-distributed and independent from $H$, which is a $GUE(n)$ matrix, i.e. $H_{ii} \sim \mathcal{N}(0,1)$ for $i = 1,...,n$, and $\Re H_{ij} = \Re H_{ji} \sim \mathcal{N}(0,1/2)$, $\Im H_{ij} = - \Im H_{ji} \sim \mathcal{N}(0,1/2)$ for $1 \leq i < j \leq n$, with entries being independent up to the Hermitian symmetry. 

Such expressions appear rather naturally when we consider powers of a unitary Brownian motion $U$, since for small $s$ it holds that $U_{t+s} \simeq (1+\sqrt{s}H)U_t$. \\

Let $C_{2j} = \{ \pi\in S_{2j}: \pi^2=\text{Id}, \forall l\in \{1,\dots, 2j\}, \pi(l)\neq l\}$, i.e. $C_{2j}$ is the set of pairings on $\{1,...,2j\}$. Then we see that
\begin{align} \label{eqn:C2j}
    &\mathbb{E} \Big( \Tr \left( H U^{\sigma_1} H U^{\sigma_2} \dots H U^{\sigma_j} \right) \overline{\Tr \left( H U^{\sigma_1} H U^{\sigma_2} \dots H U^{\sigma_j} \right)} \Big) \nonumber \\
    =& \sum_{i_1,...,i_{2j}, l_1,...,l_{2j}} \mathbb{E} \left( H_{i_1i_2} (U^{\sigma_1})_{i_2i_3} \cdots H_{i_{2j-1} i_{2j}} (U^{\sigma_j})_{i_{2j} i_1} \overline{H_{l_1l_2} (U^{\sigma_1})_{l_2l_3} \cdots H_{ l_{2j-1} l_{2j}} (U^{\sigma_j})_{l_{2j} l_1}}\right) \nonumber \\
    =& \sum_{i_1,...,i_{4j}} \mathbb{E} \left[ H_{i_1i_2} H_{i_3i_4} \cdots H_{i_{2j-1} i_{2j}} H_{i_{2j+1} i_{2j+2}} H_{i_{2j+3} i_{2j+4}} \cdots H_{i_{4j-1} i_{4j}}\right] \\
    &\hspace{2cm} \times \mathbb{E} \left[ (U^{\sigma_1})_{i_2i_3} (U^{\sigma_2})_{i_4i_5} \cdots (U^{\sigma_j})_{i_{2j}i_1} \overline{(U^{\sigma_1})_{i_{2j+1}i_{2j+4}} (U^{\sigma_2})_{i_{2j+3}i_{2j+6}} \cdots (U^{\sigma_j})_{i_{4j-1}i_{2j+2} }}\right] \nonumber \\
    =& \sum_{i_1,...,i_{4j}} \sum_{\pi\in C_{2j}} \mathbbm{1}_{\forall l\in \{1,\dots, 2j\}, (i_{2l-1},i_{2l})=(i_{2\pi(l)},i_{2\pi(l)-1} ) } \nonumber \\
    &\hspace{2cm} \times \mathbb{E} \left[ (U^{\sigma_1})_{i_2i_3} (U^{\sigma_2})_{i_4i_5} \cdots (U^{\sigma_j})_{i_{2j}i_1} \overline{(U^{\sigma_1})_{i_{2j+1}i_{2j+4}} (U^{\sigma_2})_{i_{2j+3}i_{2j+6}} \cdots (U^{\sigma_j})_{i_{4j-1}i_{2j+2} }} \right]. \nonumber
\end{align}

The condition $(i_{2l-1},i_{2l})=(i_{2\pi(l)},i_{2\pi(l)-1}) \, \, \forall l\in \{1,\dots, 2j\} \, \, \forall i_1,...,i_{4j} \in \{1,...,n\}$ allows to define a map $\pi \mapsto \tilde{\pi}$ from $C_{2j}$ to $C_{4j}$ by the formula
\begin{align*}
    \tilde{\pi}(2l-1) =& 2\pi(l), \quad \tilde{\pi}(2l) = 2 \pi(l) - 1, \quad l = 1,...,2j.
\end{align*}
Further we define the pairing $\rho \in C_{4j}$ as
\begin{align*}
    \rho := (23)(45) \cdots (2j, 1) (2j+1, 2j+4) (2j+3, 2j+6) \cdots (2j+2l-1, 2j+2l+2) \cdots (4j-1, 2j+2).
\end{align*}
See Example \ref{example:j = 2} for a list of the pairings $\pi$, $\tilde{\pi}$ and $\rho$, for $j=2$, and Figure \ref{figure:j = 2} for their depiction. \\

Note that $\rho$ and all pairings $\tilde{\pi}$ pair even numbers with odd numbers, thus $\tilde{\pi} \rho$ maps even numbers to even numbers and odd numbers to odd numbers.  Using the pairing $\tilde{\pi}$, the even numbers $i_2,i_4,...i_{4j}$ determine all the odd ones. Thus we see that
\begin{align} \label{eqn:C2j to sigma hat}
    &\sum_{i_1,...,i_{4j}} \sum_{\pi\in C_{2j}} \mathbbm{1}_{\forall l\in \{1,\dots, 2j\}, (i_{2l-1},i_{2l})=(i_{2\pi(l)},i_{2\pi(l)-1} ) } \nonumber \\
    &\hspace{2cm} \times \mathbb{E} \left[ (U^{\sigma_1})_{i_2i_3} (U^{\sigma_2})_{i_4i_5} \cdots (U^{\sigma_j})_{i_{2j}i_1} \overline{(U^{\sigma_1})_{i_{2j+1}i_{2j+4}} (U^{\sigma_2})_{i_{2j+3}i_{2j+6}} \cdots (U^{\sigma_j})_{i_{4j-1}i_{2j+2} }} \right] \nonumber \\
    =&\sum_{\pi\in C_{2j}} \sum_{i_1,...,i_{4j}} \mathbbm{1}_{\forall l\in \{1,\dots, 4j\}, i_{l} = i_{\tilde{\pi}(l)}} \mathbb{E} \left[ (U^{\sigma_1})_{i_2i_{\rho(2)}} \cdots (U^{\sigma_j})_{i_{2j}i_{\rho(2j)}} \overline{(U^{\sigma_1})_{i_{\rho(2j+4)} i_{2j+4}} \cdots (U^{\sigma_j})_{i_{\rho(2j+2)}i_{2j + 2} }} \right] \nonumber \\
    =&\sum_{\pi\in C_{2j}} \sum_{i_2,i_4,...,i_{4j}} \mathbb{E} \left[ (U^{\sigma_1})_{i_2i_{\tilde{\pi} \rho(2)}} \cdots (U^{\sigma_j})_{i_{2j}i_{\tilde{\pi} \rho(2j)}} \overline{(U^{\sigma_1})_{i_{ \tilde{\pi} \rho(2j+4)} i_{2j+4}} \cdots (U^{\sigma_j})_{i_{\tilde{\pi} \rho(2j+2)}i_{2j + 2} }} \right] \\
    =& \sum_{\pi \in C_{2j}} \mathbb{E} \left( \sum_{i_2,i_4,...i_{4j}} \prod_{l = 1}^{2j} (U^{\hat{\sigma}_{l}})_{i_{2l} i_{\tilde{\pi}\rho(2l)}} \right), \nonumber
\end{align}
where
\begin{align} \label{eqn:sigma hat}
    \hat{\sigma}_{l} = \begin{cases} \sigma_l, & l = 1,2,...,j, \\ -\sigma_{l-j-1}, & l = j + 2,...,2j, \\ -\sigma_{j}, & l = j+1. \end{cases}
\end{align}

By repeatedly applying $\tilde{\pi}\rho$ to $\{2,4,...,4j\}$, we get a partition of $\{2,4,...,4j\}$ into orbits. The set of these orbits we denote by $\mathcal{O}_{\tilde{\pi}\rho}$. We see that

\begin{align} \label{eqn:sigma hat to orbits}
\begin{split}
    & \sum_{\pi \in C_{2j}} \mathbb{E} \left( \sum_{i_2,i_4,...i_{4j}} \prod_{l = 1}^{2j} (U^{\hat{\sigma}_{2l}})_{i_{2l} i_{\tilde{\pi}\rho(2l)}} \right) \\
    =& \sum_{\pi \in C_{2j}} \mathbb{E} \left( \sum_{i_2,i_4,...i_{4j}} \prod_{o \in \mathcal{O}_{\tilde{\pi}\rho}} \prod_{w \in o} (U^{\hat{\sigma}_{w}})_{i_{w} i_{\tilde{\pi}\rho(w)}} \right) \\
    =& \sum_{\pi \in C_{2j}} \mathbb{E} \left( \prod_{o \in \mathcal{O}_{\tilde{\pi}\rho}} \Tr \left( \prod_{w \in o} U^{\hat{\sigma}_w} \right) \right) \\
    =& \sum_{\pi \in C_{2j}} \mathbb{E} \left( \prod_{o \in \mathcal{O}_{\tilde{\pi}\rho}} \Tr \left( U^{\sum_{w \in o} \hat{\sigma}_w} \right) \right).
\end{split}
\end{align}

Putting together (\ref{eqn:C2j}), (\ref{eqn:C2j to sigma hat}), (\ref{eqn:sigma hat}) and (\ref{eqn:sigma hat to orbits}), we have proven the following proposition:

\begin{proposition} \label{lemma:orbits}
Let $H$ be an $n \times n$ matrix from the $GUE(n)$, and let $U \in U(n)$ be independent and Haar-distributed. Then for $j \in \mathbb{N}$ and $\sigma_1, ..., \sigma_j \in \mathbb{N}$ it holds that
\begin{align*}
    &\mathbb{E} \Big( \Tr \left( H U^{\sigma_1} H U^{\sigma_2} \dots H U^{\sigma_j} \right) \overline{\Tr \left( H U^{\sigma_1} H U^{\sigma_2} \dots H U^{\sigma_j} \right)} \Big) \\
    =& \sum_{\pi \in C_{2j}} \mathbb{E} \left( \prod_{o \in \mathcal{O}_{\tilde{\pi}\rho}} \Tr \left( U^{\sum_{w \in o} \hat{\sigma}_w} \right) \right).
\end{align*}
\end{proposition}

\begin{example}\label{example:j = 2} For $j = 2$ we see that $\rho = (14)(23)(58)(67)$, and (see Figure \ref{figure:j = 2})
\begin{align*}
\begin{split}
    &\pi = (12)(34), \quad \tilde{\pi} = (14)(23)(58)(67), \quad \tilde{\pi} \rho = (2)(4)(6)(8),  \\
    &\pi = (13)(24), \quad \tilde{\pi} = (16)(25)(38)(47), \quad \tilde{\pi} \rho = (28)(46),  \\
    &\pi = (14)(23), \quad \tilde{\pi} = (18)(27)(36)(45), \quad \tilde{\pi} \rho = (26)(48),
\end{split}
\end{align*}
and that $\hat{\sigma}_2 = \sigma_1$, $\hat{\sigma}_4 = \sigma_2$, $\hat{\sigma}_6 = -\sigma_2$ and $\hat{\sigma}_8 = - \sigma_1$. Thus from Lemma \ref{lemma:orbits} it follows that
\begin{align*}
    &\mathbb{E} \Big( \Tr \left( H U^{\sigma_1} H U^{\sigma_2} \right) \overline{\Tr \left( H U^{\sigma_1} H U^{\sigma_2} \right)} \Big) \\
    =& \mathbb{E} \left( \Tr U^{\sigma_1} \Tr U^{\sigma_2} \Tr U^{-\sigma_2} \Tr U^{-\sigma_1} \right) \\
    &+\mathbb{E} \left( \Tr U^{\sigma_1-\sigma_1} \Tr U^{\sigma_2-\sigma_2} \right) \\
    &+\mathbb{E} \left( \Tr U^{\sigma_1-\sigma_2} \Tr U^{\sigma_2-\sigma_1} \right) \\
    =& \begin{cases} 2\sigma_1^2 + n^2 + n^2 & \sigma_1 = \sigma_2 \\ \sigma_1\sigma_2 + n^2 + |\sigma_1 - \sigma_2| & \sigma_1 \neq \sigma_ 2 \end{cases},
\end{align*}
where the last equality holds for large enough $n$ by Theorem \ref{thm:DiaconisShahshahani}.
\end{example}

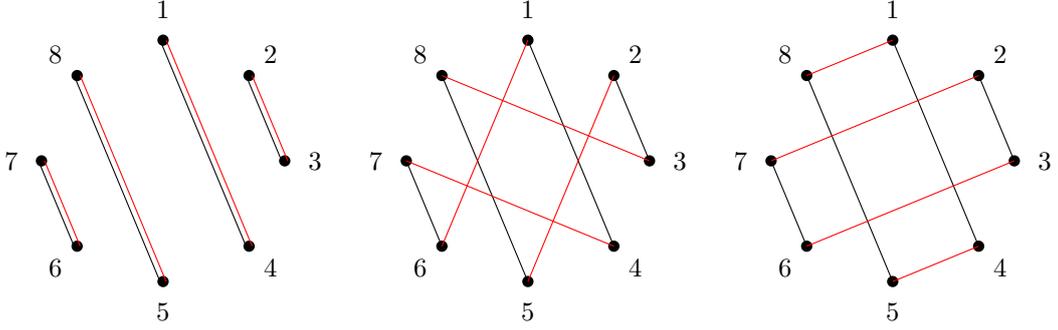
\begin{figure}[H] \label{figure:j = 2}
\center
\begin{tikzpicture}[scale = 0.8]
\node at ({2*cos(0)-6},{2*sin(0)}) [circle,fill,inner sep=1.5pt]{};
\fill ({2.5*cos(0)-6},{2.5*sin(0)}) node[] {$3$};
\node at ({2*cos(45)-6},{2*sin(45)}) [circle,fill,inner sep=1.5pt]{};
\fill ({2.5*cos(45)-6},{2.5*sin(45)}) node[] {$2$};
\node at ({2*cos(90)-6},{2*sin(90)}) [circle,fill,inner sep=1.5pt]{};
\fill ({2.5*cos(90)-6},{2.5*sin(90)}) node[] {$1$};
\node at ({2*cos(135)-6},{2*sin(135)}) [circle,fill,inner sep=1.5pt]{};
\fill ({2.5*cos(135)-6},{2.5*sin(135)}) node[] {$8$};
\node at ({2*cos(180)-6},{2*sin(180)}) [circle,fill,inner sep=1.5pt]{};
\fill ({2.5*cos(180)-6},{2.5*sin(180)}) node[] {$7$};
\node at ({2*cos(225)-6},{2*sin(225)}) [circle,fill,inner sep=1.5pt]{};
\fill ({2.5*cos(225)-6},{2.5*sin(225)}) node[] {$6$};
\node at ({2*cos(270)-6},{2*sin(270)}) [circle,fill,inner sep=1.5pt]{};
\fill ({2.5*cos(270)-6},{2.5*sin(270)}) node[] {$5$};
\node at ({2*cos(315)-6},{2*sin(315)}) [circle,fill,inner sep=1.5pt]{};
\fill ({2.5*cos(315)-6},{2.5*sin(315)}) node[] {$4$};

\draw ({2*cos(90)-6.05},{2*sin(90)}) -- ({2*cos(315)-6.05},{2*sin(315)});
\draw ({2*cos(0)-6.05},{2*sin(0)}) -- ({2*cos(45)-6.05},{2*sin(45)});
\draw ({2*cos(135)-6.05},{2*sin(135)}) -- ({2*cos(270)-6.05},{2*sin(270)});
\draw ({2*cos(180)-6.05},{2*sin(180)}) -- ({2*cos(225)-6.05},{2*sin(225)});

\draw[red] ({2*cos(90)-5.95},{2*sin(90)}) -- ({2*cos(315)-5.95},{2*sin(315)});
\draw[red] ({2*cos(0)-5.95},{2*sin(0)}) -- ({2*cos(45)-5.95},{2*sin(45)});
\draw[red] ({2*cos(135)-5.95},{2*sin(135)}) -- ({2*cos(270)-5.95},{2*sin(270)});
\draw[red] ({2*cos(180)-5.95},{2*sin(180)}) -- ({2*cos(225)-5.95},{2*sin(225)});

\node at ({2*cos(0)},{2*sin(0)}) [circle,fill,inner sep=1.5pt]{};
\fill ({2.5*cos(0)},{2.5*sin(0)}) node[] {$3$};
\node at ({2*cos(45)},{2*sin(45)}) [circle,fill,inner sep=1.5pt]{};
\fill ({2.5*cos(45)},{2.5*sin(45)}) node[] {$2$};
\node at ({2*cos(90)},{2*sin(90)}) [circle,fill,inner sep=1.5pt]{};
\fill ({2.5*cos(90)},{2.5*sin(90)}) node[] {$1$};
\node at ({2*cos(135)},{2*sin(135)}) [circle,fill,inner sep=1.5pt]{};
\fill ({2.5*cos(135)},{2.5*sin(135)}) node[] {$8$};
\node at ({2*cos(180)},{2*sin(180)}) [circle,fill,inner sep=1.5pt]{};
\fill ({2.5*cos(180)},{2.5*sin(180)}) node[] {$7$};
\node at ({2*cos(225)},{2*sin(225)}) [circle,fill,inner sep=1.5pt]{};
\fill ({2.5*cos(225)},{2.5*sin(225)}) node[] {$6$};
\node at ({2*cos(270)},{2*sin(270)}) [circle,fill,inner sep=1.5pt]{};
\fill ({2.5*cos(270)},{2.5*sin(270)}) node[] {$5$};
\node at ({2*cos(315)},{2*sin(315)}) [circle,fill,inner sep=1.5pt]{};
\fill ({2.5*cos(315)},{2.5*sin(315)}) node[] {$4$};

\draw ({2*cos(90)},{2*sin(90)}) -- ({2*cos(315)},{2*sin(315)});
\draw ({2*cos(0)},{2*sin(0)}) -- ({2*cos(45)},{2*sin(45)});
\draw ({2*cos(135)},{2*sin(135)}) -- ({2*cos(270)},{2*sin(270)});
\draw ({2*cos(180)},{2*sin(180)}) -- ({2*cos(225)},{2*sin(225)});

\draw[red] ({2*cos(90)},{2*sin(90)}) -- ({2*cos(225)},{2*sin(225)});
\draw[red] ({2*cos(45)},{2*sin(45)}) -- ({2*cos(270)},{2*sin(270)});
\draw[red] ({2*cos(0)},{2*sin(0)}) -- ({2*cos(135)},{2*sin(135)});
\draw[red] ({2*cos(315)},{2*sin(315)}) -- ({2*cos(180)},{2*sin(180)});

\node at ({2*cos(0)+6},{2*sin(0)}) [circle,fill,inner sep=1.5pt]{};
\fill ({2.5*cos(0)+6},{2.5*sin(0)}) node[] {$3$};
\node at ({2*cos(45)+6},{2*sin(45)}) [circle,fill,inner sep=1.5pt]{};
\fill ({2.5*cos(45)+6},{2.5*sin(45)}) node[] {$2$};
\node at ({2*cos(90)+6},{2*sin(90)}) [circle,fill,inner sep=1.5pt]{};
\fill ({2.5*cos(90)+6},{2.5*sin(90)}) node[] {$1$};
\node at ({2*cos(135)+6},{2*sin(135)}) [circle,fill,inner sep=1.5pt]{};
\fill ({2.5*cos(135)+6},{2.5*sin(135)}) node[] {$8$};
\node at ({2*cos(180)+6},{2*sin(180)}) [circle,fill,inner sep=1.5pt]{};
\fill ({2.5*cos(180)+6},{2.5*sin(180)}) node[] {$7$};
\node at ({2*cos(225)+6},{2*sin(225)}) [circle,fill,inner sep=1.5pt]{};
\fill ({2.5*cos(225)+6},{2.5*sin(225)}) node[] {$6$};
\node at ({2*cos(270)+6},{2*sin(270)}) [circle,fill,inner sep=1.5pt]{};
\fill ({2.5*cos(270)+6},{2.5*sin(270)}) node[] {$5$};
\node at ({2*cos(315)+6},{2*sin(315)}) [circle,fill,inner sep=1.5pt]{};
\fill ({2.5*cos(315)+6},{2.5*sin(315)}) node[] {$4$};

\draw ({2*cos(90)+6},{2*sin(90)}) -- ({2*cos(315)+6},{2*sin(315)});
\draw ({2*cos(0)+6},{2*sin(0)}) -- ({2*cos(45)+6},{2*sin(45)});
\draw ({2*cos(135)+6},{2*sin(135)}) -- ({2*cos(270)+6},{2*sin(270)});
\draw ({2*cos(180)+6},{2*sin(180)}) -- ({2*cos(225)+6},{2*sin(225)});

\draw[red] ({2*cos(90)+6},{2*sin(90)}) -- ({2*cos(135)+6},{2*sin(135)});
\draw[red] ({2*cos(45)+6},{2*sin(45)}) -- ({2*cos(180)+6},{2*sin(180)});
\draw[red] ({2*cos(0)+6},{2*sin(0)}) -- ({2*cos(225)+6},{2*sin(225)});
\draw[red] ({2*cos(315)+6},{2*sin(315)}) -- ({2*cos(270)+6},{2*sin(270)});

\end{tikzpicture}
\caption{The pairing $\rho$ is in black, the three pairings $\tilde{\pi}$ in $\tilde{C}_{8}$ are in red.}
\end{figure}

\begin{theorem}(Diaconis, Shahshahani \cite{DiaconisSha}) \label{thm:DiaconisShahshahani}
    Let $U$ be a Haar-distributed random matrix in $U(n)$ and let $Z_1,...,Z_k$ be i.i.d. standard complex Gaussian random variables. Let $a = (a_1,...,a_k)$ and $b = (b_1,...,b_k)$ with $a_j, b_j \in \mathbb{N}$, and let $n \in \mathbb{N}$ be such that
    \begin{align*}
        \max \left\{ \sum_{j = 1}^k ja_j, \sum_{j = 1}^k jb_j \right\} \leq n.
    \end{align*}
    then
    \begin{align*}
        \mathbb{E} \left( \prod_{j = 1}^k ((\Tr (U^j))^{a_j} \overline{(\Tr (U^j))^{b_j}} \right) = \delta_{ab} \prod_{j = 1}^k j^{a_j} a_j! = \mathbb{E} \left( \prod_{j = 1}^k (\sqrt{j} Z_j)^{a_j} \overline{(\sqrt{j} Z_j)^{b_j}} \right).
    \end{align*}
\end{theorem}

\section*{Acknowledgments}
Our work was supported by ERC Advanced Grant 740900 (LogCorRM). Additionally, Isao Sauzedde was funded by the EPSRC grant EP/W006227/1 during the later stage of the writing process. We are most grateful to Jon Keating and Hugo Falconet for very helpful discussions and suggestions, and to Thierry Lévy for his valuable comments.
\newpage
\bibliographystyle{IEEEtran}
\bibliography{main}

\begin{thebibliography}{10}
\providecommand{\url}[1]{#1}
\csname url@samestyle\endcsname
\providecommand{\newblock}{\relax}
\providecommand{\bibinfo}[2]{#2}
\providecommand{\BIBentrySTDinterwordspacing}{\spaceskip=0pt\relax}
\providecommand{\BIBentryALTinterwordstretchfactor}{4}
\providecommand{\BIBentryALTinterwordspacing}{\spaceskip=\fontdimen2\font plus
\BIBentryALTinterwordstretchfactor\fontdimen3\font minus
  \fontdimen4\font\relax}
\providecommand{\BIBforeignlanguage}[2]{{%
\expandafter\ifx\csname l@#1\endcsname\relax
\typeout{** WARNING: IEEEtran.bst: No hyphenation pattern has been}%
\typeout{** loaded for the language `#1'. Using the pattern for}%
\typeout{** the default language instead.}%
\else
\language=\csname l@#1\endcsname
\fi
#2}}
\providecommand{\BIBdecl}{\relax}
\BIBdecl

\bibitem{HughesKeatingOConnell}
C.~P. Hughes, J.~P. Keating, and N.~O'Connell,
  ``\BIBforeignlanguage{English}{On the characteristic polynomial of a random
  unitary matrix},'' \emph{\BIBforeignlanguage{English}{Communications in
  Mathematical Physics}}, vol. 220, 2001.

\bibitem{Spohn}
H.~Spohn, ``\BIBforeignlanguage{English}{Dyson's model of interacting
  {B}rownian motions at arbitrary coupling strength},''
  \emph{\BIBforeignlanguage{English}{Markov Processes and Related Fields}},
  vol.~4, 1998.

\bibitem{BourgadeFalconet}
P.~Bourgade and H.~Falconet, ``Liouville quantum gravity from random matrix
  dynamics,'' 2022, ar{X}iv 2206.03029.

\bibitem{Dyson}
F.~J. Dyson, ``A {Brownian}-motion model for the eigenvalues of a random
  matrix,'' \emph{Journal of Mathematical Physics}, vol.~3, 1962.

\bibitem{Webb2}
C.~Webb, ``{The characteristic polynomial of a random unitary matrix and
  Gaussian multiplicative chaos - The $L^2$-phase},'' \emph{Electronic Journal
  of Probability}, vol.~20, 2015.

\bibitem{Conreyetall}
J.~B. Conrey, D.~W. Farmer, J.~P. Keating, M.~O. Rubinstein, and N.~C. Snaith,
  ``{Integral Moments of $L$-Functions},'' \emph{Proceedings of the London
  Mathematical Society}, vol.~91, 2005.

\bibitem{FyodorovHiaryKeating}
Y.~V. Fyodorov, G.~A. Hiary, and J.~P. Keating, ``{Freezing Transition,
  Characteristic Polynomials of Random Matrices, and the Riemann Zeta
  Function},'' \emph{Physical Review Letters}, vol. 108, 2012.

\bibitem{FyodorovKeating}
Y.~V. Fyodorov and J.~P. Keating, ``{Freezing Transitions and Extreme Values:
  Random Matrix Theory, $\zeta (1/2+it)$, and Disordered Landscapes},''
  \emph{Philosophical Transactions of the Royal Society A: Mathematical,
  Physical and Engineering Sciences}, vol. 372, 2014.

\bibitem{KeatingSnaith}
J.~P. Keating and N.~Snaith, ``{Random Matrix Theory and $\zeta (1/2+it)$},''
  \emph{Communications in Mathematical Physics}, vol. 214, 2000.

\bibitem{KeatingSnaith2}
------, ``{Random Matrix Theory and {$L$}-functions at $s= 1/2$},''
  \emph{Communications in Mathematical Physics}, vol. 214, 2000.

\bibitem{BaileyKeating2}
E.~C. Bailey and J.~P. Keating, ``Maxima of log-correlated fields: some recent
  developments,'' \emph{Journal of Physics A: Mathematical and Theoretical},
  vol.~55, 2022.

\bibitem{Diaconis}
P.~Diaconis and S.~N. Evans, ``Linear functionals of eigenvalues of random
  matrices,'' \emph{Transactions of the American Mathematical Society}, vol.
  353, 2001.

\bibitem{DiaconisSha}
P.~Diaconis and M.~Shahshahani, ``On the eigenvalues of random matrices,''
  \emph{Journal of Applied Probability}, vol. 31A, 1994.

\bibitem{Dobler}
C.~D\"{o}bler and M.~Stolz, ``Stein's method and the multivariate {CLT} for
  traces of powers on the classical compact groups,'' \emph{Electronic Journal
  of Probability}, vol.~16, 2011.

\bibitem{johansson2020}
K.~Johansson and G.~Lambert, ``{Multivariate normal approximation for traces of
  random unitary matrices},'' \emph{The Annals of Probability}, vol.~49, 2021.

\bibitem{Webb}
C.~Webb, ``Linear statistics of the circular {$\beta$}-ensemble, {S}tein's
  method, and circular {D}yson {B}rownian motion,'' \emph{Electronic Journal of
  Probability}, vol.~21, 2016.

\bibitem{Berestycki}
N.~Berestycki, ``An elementary approach to {G}aussian multiplicative chaos,''
  \emph{Electronic Communications in Probability}, vol.~22, 2017.

\bibitem{Hitchhiker}
E.~{Di Nezza}, G.~Palatucci, and E.~Valdinoci, ``Hitchhiker's guide to the
  fractional {S}obolev spaces,'' \emph{Bulletin des Sciences Mathématiques},
  vol. 136, 2012.

\bibitem{Kubrusly}
C.~S. Kubrusly and N.~Levan, ``\BIBforeignlanguage{English}{Preservation of
  tensor sum and tensor product},'' \emph{\BIBforeignlanguage{English}{Acta
  Mathematica Universitatis Comenianae, New Series}}, vol.~80, 2011.

\end{thebibliography}

\end{document}